\documentclass[11pt,reqno]{amsart}

\usepackage{amsmath}
\usepackage{amssymb}
\usepackage{amsthm}
\usepackage{amscd, amsfonts, mathrsfs}

\usepackage{cases}
\usepackage{xfrac}

\usepackage[dvips]{epsfig}
\usepackage{verbatim} 

\usepackage{epsf}
\usepackage[bookmarksnumbered,pdfpagelabels=true,plainpages=false,colorlinks=true,
            linkcolor=black,citecolor=black,urlcolor=blue]{hyperref}

\theoremstyle{plain}
\newtheorem{theorem}{Theorem}[section]
\newtheorem{lemma}[theorem]{Lemma}
\newtheorem{prop}[theorem]{Proposition}

\theoremstyle{definition}
\newtheorem{rema}[theorem]{Remark}
\newtheorem{notation}[theorem]{Notation}
\newtheorem{defi}[theorem]{Definition}

\numberwithin{equation}{section}

\newenvironment{ sys_eq }{\ left \ lbrace \ begin { array }{ @ {} l@ {}}}{\ end { array }\ right .}

\newcommand{\cA}{\mathcal A}

\newcommand{\cH}{\mathcal H}

\newcommand{\cL}{\mathcal L}

\newcommand{\cR}{\mathcal R}

\newcommand{\al}{\alpha}
\newcommand{\be}{\beta}

\newcommand{\Ga}{\Gamma}

\newcommand{\si}{\sigma}

\newcommand{\Om}{\Omega}

\newcommand{\RR}{\mathbb R}

\newcommand{\rar}{\rightarrow}

\newcommand{\id}{\operatorname{id}}
\newcommand{\dive}{\operatorname{div}}

\newcommand{\p}{\parallel}

\def\Xint#1{\mathchoice
{\XXint\displaystyle\textstyle{#1}}%
{\XXint\textstyle\scriptstyle{#1}}%
{\XXint\scriptstyle\scriptscriptstyle{#1}}%
{\XXint\scriptscriptstyle\scriptscriptstyle{#1}}%
\!\int}
\def\XXint#1#2#3{{\setbox0=\hbox{$#1{#2#3}{\int}$}
\vcenter{\hbox{$#2#3$}}\kern-.5\wd0}}

\def\dashint{\Xint-}

\title[Linear dynamic boundary]{On a linear problem arising in dynamic boundaries}

\author[Disconzi]{Marcelo M. Disconzi}
\address{Department of Mathematics\\
Vanderbilt University, Nashville, TN 37240, USA}
\email{marcelo.disconzi@vanderbilt.edu}
\thanks{The author is partially supported by NSF grant 1305705.}

\begin{document}

\begin{abstract}
We study a linear problem that arises in the study of dynamic boundaries, 
in particular in free boundary problems in connection with fluid dynamics.
The equations are also very natural and of interest on their own.
\end{abstract}

\maketitle

\tableofcontents

\section{Introduction. \label{intro}}
Consider the problem
\begin{subnumcases}{\label{linear_problem}} 
 \Delta f = 0 & in $\Om$, \label{harmonic_ext_linear} \\
 \ddot{f} - \kappa 
\overline{\Delta} \partial_\nu f  =  G  & on $\partial \Om$, \label{eq_f_bry_estimates} \\
f(0, \cdot) = f_0, \, \dot{f}(0, \cdot ) = f_1 & on $\partial \Om$ \label{ic_f_bry_estimates},
\end{subnumcases}
where $\Om \subset \RR^n$ is a bounded domain whose boundary $\partial \Om$
 is an 
$n-1$-dimensional manifold embedded in $\RR^n$;
$\overline{\Delta}$ is the Laplacian on $\partial \Om$ and $\Delta$ the Laplacian in 
$\RR^n$; $\partial_\nu$ is the outer normal derivative on $\partial \Om$; 
$\kappa$ is a positive constant; $f: [0,T] \times \Om \rar \RR$ is the 
unknown, $T>0$; $G: [0,T] \times \partial \Om \rar \RR$, $f_0: \partial \Om \rar \RR$, 
and $f_1:\partial \Om \rar \RR$ are  given functions; 
and `` $\dot{}$ " means derivative with respect
to $t$, where we write $f = f(t,x)$, $t \in [0,T]$, $x \in \overline{\Om}$.
We shall elaborate an appropriate notion of weak solution to (\ref{linear_problem}),
 then establish the existence and uniqueness of weak solutions on any 
time interval $[0,T]$.

Let us discuss some motivations to study (\ref{linear_problem}). In this regard, 
it is perhaps worthwhile to start noticing that, from a PDE perspective, 
problem (\ref{linear_problem}) is 
 very natural. Without the normal derivative $\partial_\nu$ on the second
term on the left-hand side of (\ref{eq_f_bry_estimates}), the problem decouples: 
(\ref{eq_f_bry_estimates})-(\ref{ic_f_bry_estimates})
becomes a wave equation on the boundary, which can be solved by standard techniques, and
equation (\ref{harmonic_ext_linear}) says that this solution on $\partial \Om$ is extended to 
the interior via the unique harmonic extension of $\left. f \right|_{\partial \Om}$.
A similar procedure is no longer possible when the term $\partial_\nu$
is present, as in (\ref{linear_problem}). The introduction of the normal derivative
can be viewed as
one of the simplest ways of  modifying the wave equation on the boundary as to make 
it dependent on the interior values of $f$. 

A more direct motivation to investigate (\ref{linear_problem}) is
that it arises at the linearized level in the study of the (incompressible) 
free boundary Euler equations, as we now explain.

Consider the motion of an inviscid incompressible fluid 
within a bounded region of space, and suppose further that the
boundary of the region confining the fluid is not rigid, being
allowed to move according to the pressure exerted by the fluid
(hence
the name ``free bounary").
This is the situation, for example, in a liquid drop, or in 
Newtonian self-gravitating fluid bodies, such as
stars \cite{LindNor, Mak, Nis, White} (the analogue problem for viscous fluids
was first and extensively studied by Solonnikov \cite{MogSol, Sol1, Sol2, Sol3, Sol3,
Sol5, Sol6, Sol7}, with some more recent
advances found in 
\cite{KPW, PS, Se1, Se2, Se3} and references therein).
In such situations, the domain containing the fluid  changes over time. 
One thus
writes $\Omega(t)$, and $\Omega(t)$ becomes one of the unknowns of the problem.

The equations of motion describing the situation of the previous paragraph are
the well-known free boundary Euler equations:
\begin{gather}
\begin{cases}
\frac{\partial u}{\partial t} + \nabla_u u = - \nabla p & \text{ in } \Om(t), \\
\dive (u) = 0 & \text{ in } \Om(t),  \\
p = \kappa \cA & \text{ on  } \partial \Om(t), \\
\langle u, \nu \rangle = v & \text{ on  } \partial \Om(t), \\
u(0) = u_0,
\end{cases}
\label{free_Euler}
\end{gather}
where $u$ is the fluid velocity, $p$ is the fluid pressure,
 $\cA$ is the mean curvature
of $\partial \Om(t)$, $\nu$ is the unit outer normal to $\partial \Om(t)$, 
$v$ is the velocity of the moving boundary $\partial \Om(t)$,  and $\kappa$ a 
non-negative constant known as coefficient
of surface tension. We refer the reader to the literature (e.g. \cite{CS, Lin, White})
for a detailed discussion of these equations. It is important to point out that,
despite its importance and the great deal of work dedicated 
to (\ref{free_Euler}) 
\cite{Amborse, AmbroseMasmoudi, ChLin, Craig, E0, Lannes, Lin2, Nalimov, Sch, 
ShatahZeng, Wu, Yosihara}, only 
recently the problem has been shown to be well-posed \cite{CS, CSB, Lin}
(other recent results, including the study of the compressible free boundary
Euler equations, are \cite{CS2, CS3, CS4, CSH, CSL}).

A very natural question is that of the behavior of solutions to (\ref{free_Euler})
in the limit $\kappa \rar \infty$. Physically, large values of $\kappa$ correspond
to domains  with longer relaxation times or, more colloquially, to stiffer domains.
Therefore, one would expect that solutions (\ref{free_Euler}) with large $\kappa$ 
should be near solutions of the standard Euler equations in the \emph{fixed} 
domain $\Om \equiv \Om(0)$. 

The study of the limit $\kappa \rar \infty$, along with a proof of the corresponding convergence,
 was carried out by the author
and David G. Ebin in \cite{DE2d} in the case of two-spatial dimensions (the
reader is also referred to \cite{DE2d} for a more detailed discussion of
the intuition behind this convergence).
The core of the analysis
consists in studying the problem from the point of view of 
Lagrangian coordinates, in which case all quantities can be written as time-dependent functions
on the fixed domain $\Om$ ($=\Om(0)$). The flow of the vector 
field $u$, $\eta(t, \cdot): \Om \rar \RR^n$,
is decomposed in a part fixing the boundary and a boundary motion. Such 
decomposition takes the form
\begin{gather}
\eta = (\id + \nabla f) \circ \beta,
\label{decomp}
\end{gather}
where $\beta$ is a diffeomorphism of the domain $\Om$ (so in particular $\beta(\partial \Om) =
\partial \Om$), $f: \Om \rar \RR$, and $\id$ is 
the identity diffeomorphism. The term $\nabla f$ controls the motion of the boundary, and 
using (\ref{free_Euler}), it is possible to derive an equation for $f$. 
\emph{At the linearized level and to highest order,
this equation is} (\ref{linear_problem}); the third order operator
$\overline{\Delta} \partial_\nu$ stems from the mean curvature of the 
moving boundary. See \cite{DE2d} for details.

In \cite{DE2d}, we were interested in studying the limit $\kappa \rar \infty$, 
and therefore we relied on the aforementioned 
existence results for (\ref{free_Euler}) (particularly, \cite{CS}). Therefore,
the existence of solutions for the linearized problem, namely,  (\ref{linear_problem}), has not been addressed
in  \cite{DE2d}.
While it will be shown in a future work that the decomposition (\ref{decomp}) can be employed
to derive existence of solutions to (\ref{free_Euler}) \cite{DE3d}, such an analysis
is based on the calculus of pseudo-differential operators and techniques similar to 
\cite{K}. Hence, the simpler, more traditional
 methods that we shall present here do not appear
elsewhere. Furthermore, the results in \cite{DE3d} do not cover the case of 
\emph{weak} solutions to
(\ref{linear_problem}), which is the main point of this paper (see definition \ref{defi_weak}). 

We point out that the singular limit $\kappa \rar \infty$ investigated in 
\cite{DE2d,DE3d} fits into the larger picture of properties of solutions viewed
as curves on infinite dimensional manifolds of mappings, which has been extensively
studied in the context of the Euler 
equations. See the references \cite{BB, E_manifold, E0, E1, E2, E3, E4, EbinSymp, ED, EM, MEF},
and the discussion in the introduction of \cite{DE2d}. While here we shall not study
the dependence on the parameter $\kappa$,  it is instructive to keep the above ideas in mind.
In this regard, compare (\ref{linear_problem}) with the toy-model presented in \cite{E2}. 

Naturally, dynamic boundary value problems have a long history, 
leading to  variants of (\ref{linear_problem}). Adding to the 
aforementioned works, whose focus is mainly on equations of hyperbolic type,
the reader can consult, for instance, \cite{Escher, Hintermann}
and references therein, for a point of view that stresses parabolic equations.
Equations involving two time derivatives and a
third order operator also have been studied before (see, for instance,
\cite{Glenn}, and references therein, and see also the related
\cite{TG}). In particular, due to the elliptic operator $\overline{\Delta}$,
the boundary equation
(\ref{eq_f_bry_estimates}) is reminiscent of the so-called Wentzell boundary
conditions, which have been widely studied by 
A. Favini, G. Goldstein, J. Goldstein, and S.  Romanelli
 (a sample of such works is \cite{FGGR1, FGGR2, FGGR3, FGGR4, FGGR5}).

In order to state our results, some notation and definitions are needed.
Fix some $ T > 0$. Define
\begin{gather}
X^3_T(\partial \Om) = C^0([0,T], H^3(\partial \Om) ) 
\cap C^1([0,T], H^{\frac{3}{2}}(\partial \Om) )
\cap C^2([0,T], H^{0}(\partial \Om) ),
\nonumber
\end{gather}
and 
\begin{gather}
X^\frac{3}{2}_T(\partial \Om) = C^0([0,T], H^\frac{3}{2}(\partial \Om) ) 
\cap C^1([0,T], H^{0}(\partial \Om) ),
\nonumber
\end{gather}
with norms
\begin{align}
\p f \p_{X^3_T(\partial \Om)} & = \sup_{t \in [0,T]} \p f \p_{3, \partial}
+ \sup_{t \in [0,T]}  \p \dot{f} \p_{\frac{3}{2}, \partial}
+ \sup_{t \in [0,T]} \p \ddot{f} \p_{0, \partial},
\nonumber
\end{align}
and
\begin{align}
\p f \p_{X^\frac{3}{2}_T(\partial \Om)} & = \sup_{t \in [0,T]} \p f \p_{\frac{3}{2}, \partial}
+ \sup_{t \in [0,T]}  \p \dot{f} \p_{0, \partial}.
\nonumber
\end{align}
Above, $H^s(\partial \Om)$ is the 
Sobolev space whose norm
is denoted by $\p \cdot \p_{s,\partial}$. 
Notice that 
\begin{gather}
X_T^3(\partial \Om) \subset X_T^{\frac{3}{2}}(\partial \Om).
\nonumber
\end{gather}
The intersections forming $X_T^s(\partial \Om)$ are of 
Sobolev spaces that differ by $\frac{3}{2}$ derivatives. This is because equation 
(\ref{eq_f_bry_estimates}) is second order in time and third order in space, thus each 
time derivative corresponds to $\frac{3}{2}$ spatial derivatives.
The spaces
$X_T^s(\Om)$ are similarly defined, and the norm in $H^s(\Om)$ is denoted $\p \cdot \p_s$.

As it is implied in the above definitions, we are working with Sobolev spaces 
defined with $s \in \RR$. In the case $s \geq 0$, it is useful to have the following
explicit form. Put $s = m + \si$, where $m$ is an integer and 
$0 < \si < 1$. Then
\begin{gather}
\p u \p_s^2 = \p u \p^2_m + \left[ D^m u \right]_\si^2,
\label{fractional_norm}
\end{gather}
where $\left[ \cdot \right]_\si$ is the semi-norm
\begin{gather}
\left[  u \right]_\si^2 = \int_{\Om \times\Om}
\frac{ |u(x) - u(y) |^2 }{ |x-y|^{n+2\si}} \, dx dy, 
\nonumber
\end{gather}
with  $n$ as the dimension of $\Om$, and $ \p \cdot \p^2_m$ as  the standard Sobolev
norm defined for integer $m$ \cite{Hitch}. As usual, $H^0$ is simply the $L^2$ space.

As (\ref{linear_problem}) has not appeared in 
the literature before, our main 
interest is to define a natural notion of weak solutions to problem (\ref{linear_problem}), and
then show that these solutions exist. With this in mind, our treatment 
will focus on the simple situation where $\Om$ is the unit ball, 
and we restrict ourselves to the case $n=3$. 
In this situation, 
problem (\ref{linear_problem}) simplifies considerably, although many 
of the arguments below can be extended to a more general setting. 
Furthermore, this covers one of
the main cases of interest, namely, that motivated by the linearization of 
(\ref{free_Euler}) as discussed above and studied in \cite{DE2d, DE3d}.
We now proceed to state our results.

Denote $L^2(T) = L^2 ( [0,T] \times \partial \Om )$. Define a map
\begin{align}
\begin{split}
& \cL: X^3_T(\partial \Om) \rar L^2(T) \times H^\frac{3}{2}(\partial \Om) \times H^0(\partial \Om)
\equiv \cH  \\
& \cL(f) = (\ddot{f} - \kappa \overline{\Delta} \partial_\nu f, f(0), \dot{f}(0)),
\end{split}
\nonumber
\end{align}
where $\partial_\nu f$ is computed using
the harmonic extension of $f$ to $\Om$,
and we write $f(0) = f(0,\cdot)$, $\dot{f}(0) = \dot{f}(0,\cdot)$. Let $\cR \subset \cH$
be the image of $\cL$. We shall prove the following.

\begin{prop}
Let $\Om \subset \RR^3$ be the ball of radius one centered at the origin, 
and fix some $T>0$.
Let $\cL$, $\cR$, and $\cH$ be as above. Then:\\

\noindent (i) $\cL$ is injective, and since 
$X^3_T(\partial \Om) \subset X^\frac{3}{2}_T(\partial \Om)$, 
$\cL^{-1}$ defines a map
\begin{gather}
\cL^{-1}: \cR \rar X^\frac{3}{2}_T(\partial \Om).
\nonumber
\end{gather}
The map $\cL^{-1}$ is continuous as a map from $\cR$ to $X^\frac{3}{2}_T(\partial \Om)$. \\

\noindent (ii) The closure of $\cR$ in $\cH$, denoted $\overline{\cR}$, 
is the whole of $\cH$, i.e., 
$\overline{\cR} = \cH$, and $\cL^{-1}$ extends to a continuous
linear map, $\overline{\cL^{-1}}$, from $\cH$ to $X^\frac{3}{2}_T(\partial \Om)$. \\

\noindent (iii) The image of $\overline{\cL^{-1}}$ is
\begin{gather}
\overline{X^3_T(\partial \Om)\,}{}^{X^\frac{3}{2}_T(\partial \Om)},
\nonumber
\end{gather}
i.e., the closure of $X^3_T(\partial \Om) \subset X^\frac{3}{2}_T(\partial \Om)$
in the $X^\frac{3}{2}_T(\partial \Om)$ topology.
\label{prop_extension}
\end{prop}

We can now introduce the following.
\begin{defi}
Let $f_0 \in H^\frac{3}{2}(\partial \Om)$, $f_1 \in H^0(\partial \Om)$, and $G \in L^2(T)$ be given.
We say that 
\begin{gather}
f \in \overline{X^3_T(\partial \Om)\,}{}^{X^\frac{3}{2}_T(\partial \Om)} \bigcap
X^2_T( \Om)
\nonumber
\end{gather}
is a weak solution
of (\ref{linear_problem}), if $\left. f \right|_{\partial \Om} = \overline{\cL^{-1}}\left( (G, f_0, f_1 )\right)$,
where $\overline{\cL^{-1}}: \cH \rar \overline{X^3_T(\partial \Om)\,}{}^{X^\frac{3}{2}_T(\partial \Om)}$
is the map given by proposition \ref{prop_extension}, and 
$f$ satisfies (\ref{harmonic_ext_linear}) in $\Om$. 
\label{defi_weak}
\end{defi}

To understand why this is a suitable definition of weak solutions for problem (\ref{linear_problem}),
 one should think of the example 
of the wave equation. In that case, given initial data in $H^1(\RR^n) \times H^{0}(\RR^n)$ and an inhomogeneous
term in, say, $L^1([0,T], H^{0}(\RR^n))$, the weak solution $u$ is in 
$C^0([0,T],H^1(\RR^n)) \cap C^1([0,T], H^{0}(\RR^n))$. Thus, the weak solution
has one less spatial derivative than the order of the equation, with $\partial_t u$
one degree less differentiable in space than $u$ itself. Such a regularity is
a consequence of the energy estimate, in which an integration by parts is performed.
In our case, each time derivative corresponds to $\frac{3}{2}$ spatial ones, and
we heuristically think of integrating by parts half of the derivatives of the 
third order spatial term.

Proposition \ref{prop_extension} essentially contains the existence of weak solutions, 
but we state it separately for convenience.

\begin{theorem}
Let $\Om \subset \RR^3$ be the ball of radius one centered at the origin, 
and fix some $T>0$. Given $G \in L^2(T)$, $f_0 \in H^\frac{3}{2}(\partial \Om)$,
and $f_1 \in H^0(\partial \Om)$, there exists a unique weak solution to the 
problem (\ref{linear_problem}). 
\label{main_theorem}
\end{theorem}

It should be stressed that for sufficiently regular data, existence for
(\ref{linear_problem}) can probably be derived by other  means. 
The novelty
of theorem \ref{main_theorem} is centered around the notion  of weak
solutions and their existence. In this regard, it is important to stress
that a semi-group approach can also be employed to study (\ref{linear_problem}),
in which case one is led, via Stone's theorem, to investigate the existence
of mild-solutions to the problem \cite{ABHN, Yosida}. Such mild solutions are,
 in fact, closely
related to our notion of weak solution. We believe, however, that the energy-method
approach here employed is significantly simpler in the sense that it does not rely
on heavy functional-analytic techniques, and is also of independent  interest 
to the community more acquainted with such type of estimates.

\section{Energy estimates.}

In this section we carry out the necessary energy estimates for the proofs
 of proposition \ref{prop_extension}
and theorem \ref{main_theorem}. We start recalling some useful tools and fixing some of 
the notation.

\subsection{Auxiliary results.\label{auxiliary}}

Here, we collect some well known facts that will be used in
the paper. Their proofs can be found in many sources, e.g., 
\cite{Adams, BB, Hitch, E1, P}.

First, recall that restriction to the boundary gives rise to 
a bounded linear map,
\begin{gather}
 \p u \p_{s, \partial} \leq C \p u \p_{s + \frac{1}{2}},~~ s > 0,
\label{restriction}
\end{gather}
with $C =C(n,s,\Om)$.

The usual interpolation inequality will also be needed: if $s_1 < s_2 < s_3$, then
\begin{gather}
\p u \p_{s_2} \leq \p u \p_{s_1}^\frac{s_3 - s_2}{s_3-s_1} \p u \p_{s_3}^\frac{s_2-s_1}{s_3-s_1}.
\label{interpolation}
\end{gather}

We finally recall the standard Cauchy inequality with $\gamma$,
\begin{gather}
ab \leq \gamma a^2 + \frac{1}{4\gamma} b^2,
\label{Cauchy_epsilon}
\end{gather}
$\gamma > 0$ (this inequality is usually called Cauchy inequality with 
$\varepsilon$,
with the letter $\varepsilon$ used instead of $\gamma$. We shall reserve 
$\varepsilon$ for other purposes below, thus we use $\gamma$ in (\ref{Cauchy_epsilon}) to avoid confusion).

\subsection{Coordinates and notation.}

Here, we make some remarks about coordinates and notation.

Recalling that $\Om$ is the ball of radius one centered at the origin, we write 
\begin{gather}
\partial \Om_r = \partial B_r(0).
\nonumber
\end{gather}
Sometimes, we employ spherical coordinates $(r,\phi,\theta)$, so that 
\begin{gather}
\Delta = \partial_r^2 + \frac{2}{r} \partial_r + \frac{1}{r^2} \Delta_{S^2},
\label{Laplacian_spherical}
\end{gather}
where $\Delta_{S^2}$ is the Laplacian on the standard round sphere, given
 in these coordinates
by
\begin{gather}
\Delta_{S^2} =   \partial_\phi^2 + \frac{\cos \phi}{\sin \phi} \partial_\phi 
+ \frac{1}{\sin^2 \phi} \partial^2_\theta.
\nonumber
\end{gather}
In particular, 
the Laplacian on $\partial \Om_r$, which we denote $\overline{\Delta}$ for any $r$, is
 \begin{gather}
\overline{\Delta} = \frac{1}{r^2} \Delta_{S^2} .
\label{induced_Laplacian_spherical}
\end{gather}
We shall use $\overline{\Delta}$ as an operator on the whole of $\Om$. 
To be precise,  this is not defined
at zero, but the origin can be removed  without changing the value of the integrals
$\int_\Om$ containing $\overline{\Delta}$ that will appear below. 
In particular, $\overline{\Delta} f$ is defined
 on $\Om$ (but the origin).

We shall also make use of the following coordinate choice.
For $\varepsilon > 0$,  let
\begin{gather}
\Om_\varepsilon = \Om \backslash (B_\varepsilon(0) \cup C_\varepsilon ),
\nonumber
\end{gather}
where $C_\varepsilon$ is the cone given in spherical  coordinates by
$\{ \phi \geq \pi - \varepsilon \}$.
Choose Fermi  coordinates $\{ x^\mu \}_{\mu=1}^3$ at the north 
 pole of $\partial \Om$. These coordinates
cover $\Om_\varepsilon$, and the Euclidean metric takes the form
$g = (g_{\al \be})$, with 
 $g_{33} = 1$, $g_{i3} = 0$, $i=1,2$, and $g_{ij}$, $i,j=1,2$, 
 being the metric induced on 
 the 
level sets $\{ x^3 = \text{ constant} \}$, which in turn correspond 
to  $\partial \Om_r \cap \Om_\varepsilon$.
 Furthermore, $\partial_3$ is orthogonal to  $\partial \Om_r \cap \Om_\varepsilon$,
 and $\partial_3 = - \partial_r$. We illustrate the construction of these coordinates
 in figure \ref{figure_Om_e}, where we also depict further notation 
 that will be used below.
 
\begin{figure}[!ht]
\centering
 \rotatebox{0.0}{\includegraphics[scale=.45]{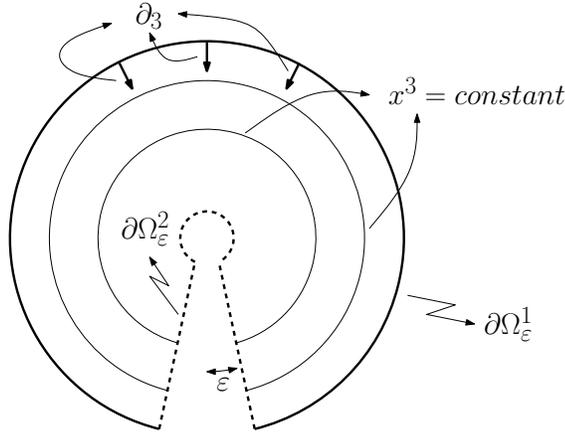}}
 \caption{Illustration of the set $\Om_\varepsilon$ and its 
 boundary $\partial \Om_\varepsilon = \partial \Om_\varepsilon^1 \cup 
 \partial \Om_\varepsilon^2$.}
 \label{figure_Om_e}
\end{figure}
 
For the rest of the paper, the following convention is adopted.
\begin{notation}
Greek indices run from $1$ to $3$ and Latin indices from $1$ to $2$. The letter
$C$ will be used to denote several different constants, as usual.
\end{notation}

In the above coordinates, equation  (\ref{harmonic_ext_linear}) then reads 
\begin{gather}
\nabla^\mu \nabla_\mu f = 0,
\nonumber
\end{gather}
where $\nabla$ is  covariant differentiation in the Euclidean metric, but written in 
this system of coordinates. Notice that all covariant derivatives will commute, as 
the metric is flat, and we shall use this in the calculations below.

\subsection{Basic energy inequality.\label{section_basic_energy}}

For the rest of this section, let $f \in X^3_T(\partial \Om)$ 
be a solution to (\ref{eq_f_bry_estimates})-(\ref{ic_f_bry_estimates}).
We also denote by $f$ its harmonic extension to $\Om$, i.e., 
$f$ satisfies (\ref{harmonic_ext_linear}) in $\Om$.

Define the energy
\begin{gather}
E = \frac{1}{2} \int_{\partial \Om} \left( \dot{f}^2
- \kappa f\overline{\Delta} \partial_\nu f \right).
\label{def_E}
\end{gather}

Differentiating
and integrating by parts,
\begin{gather}
\dot{E} = \int_{\partial \Om} \dot{f} \ddot{f} -
\frac{1}{2}\kappa \int_{\partial \Om} \dot{f}    \overline{\Delta} \partial_\nu f
-
 \frac{1}{2}\kappa \int_{\partial \Om}   \overline{\Delta} f \partial_\nu  \dot{f}.
\label{E_1}
\end{gather}
Let $\varphi = \varphi(r)$ be a sufficient regular function such that
$\varphi(1) = 1$, i.e.,  $\varphi \equiv 1$ on $\partial \Om$.
Apply   
Green's identity
\begin{gather}
\int_\Om \left( u \Delta v - v \Delta u \right) = \int_{\partial \Om}
\left( u \frac{\partial v }{\partial \nu} - v \frac{\partial u}{\partial \nu} \right)
\nonumber
\end{gather}
with $v = \varphi \overline{\Delta} f$ (which equals $\overline{\Delta} f$ on $\partial \Om$)
 and $u = \dot{f}$, to get
\begin{align}
\begin{split}
\int_{\partial \Om} \overline{\Delta } f \partial_\nu \dot{f}
=
\int_{\partial \Om} \varphi \overline{\Delta } f \partial_\nu \dot{f}
& = 
\int_{\partial \Om} \dot{f}  \overline{\Delta } \partial_\nu f +
\int_{\partial \Om} \dot{f} [\partial_\nu, \varphi \overline{\Delta} ] f
-
\int_\Om   \dot{f} \Delta 
(\varphi \overline{\Delta} f ) ,
\end{split}
\label{E_2}
\end{align}
where $[\cdot, \cdot]$ is the commutator, we have used that (\ref{harmonic_ext_linear}) implies $\Delta \dot{f} = 0$, and that $\varphi = 1$ on $\partial \Om$.
Using (\ref{E_2}) into (\ref{E_1}) gives
\begin{gather}
\dot{E} = \int_{\partial \Om} \dot{f} \left( \ddot{f} -
\kappa \overline{\Delta} \partial_\nu f \right )
+ r_1
\label{dot_E}
\end{gather}
where 
\begin{gather}
r_1 = -\frac{1}{2} \kappa 
\int_{\partial \Om} \dot{f} [\partial_\nu, \varphi \overline{\Delta} ] f
+ \frac{1}{2} \kappa
\int_\Om   \dot{f} \Delta 
(\varphi \overline{\Delta} f ) .
\nonumber
\end{gather}
To analyze $r_1$, pick $\varphi(r) = r^2$, which satisfies the previous assumptions on
$\varphi$. Then, on $\partial \Om$,
\begin{gather}
\partial_\nu (\varphi\overline{\Delta} f ) = \partial_r(r^2 \overline{\Delta} f )
= \partial_r \Delta_{S^2} f = \Delta_{S^2} \partial_r f =
\overline{\Delta} \partial_\nu f = \varphi \overline{\Delta} \partial_\nu f,
\nonumber
\end{gather} 
which implies $[\partial_\nu, \varphi \overline{\Delta} ] = 0$.
Using also (\ref{Laplacian_spherical}),
\begin{align}
\begin{split}
\Delta (\varphi \overline{\Delta} f) & = \Delta \Delta_{S^2} f 
 = \Delta_{S^2} \partial_r^2 f
+ \Delta_{S^2}\left( \frac{2}{r} \partial_r f \right)
+ \Delta_{S^2} \left (\frac{1}{r^2} \Delta_{S^2} f \right)  
= \Delta_{S^2} \Delta f = 0.
\end{split}
\nonumber
\end{align}
Thus $r_1 = 0$ and  (\ref{dot_E}) becomes
\begin{gather}
\dot{E} = \int_{\partial \Om} \dot{f} \left( \ddot{f} -
\kappa \overline{\Delta} \partial_\nu f \right ).
\label{dot_E_2}
\end{gather}
Invoking  (\ref{eq_f_bry_estimates}) and the Cauchy-Schwarz inequality, 
 (\ref{dot_E_2}) gives
\begin{gather}
\dot{E} 
\leq \frac{1}{2} \p \dot{f} \p_{0,\partial}^2 + \frac{1}{2} \p G \p_{0,\partial}^2,
\label{dot_E_CS}
\end{gather}
where we recall that $\p \cdot \p_{s,\partial}$ is the Sobolev
norm on the boundary.
\begin{lemma}
\begin{gather}
-\int_{\partial \Om} \overline{\Delta} f \partial_\nu f \geq 0.
\nonumber
\end{gather}
\label{positivity_lemma}
\end{lemma}
\begin{proof}
Integrating by parts,
\begin{gather}
-\int_{\partial \Om} \overline{\Delta} f \partial_\nu f
= \int_{\partial \Om} \langle \nabla_\partial f, \nabla_\partial \partial_\nu f \rangle,
\label{first_by_parts_positivity}
\end{gather}
where $\nabla_\partial$ is the gradient on $\partial \Om$.
We need to commute $\nabla_\partial$ and 
$\partial_\nu$.
This is easily done in spherical coordinates, yielding
\begin{gather}
\left[ \nabla_\partial, \nabla_\nu \right] = \frac{1}{r} \nabla_\partial .
\nonumber
\end{gather}
(\ref{first_by_parts_positivity}) becomes
\begin{align}
\begin{split}
-\int_{\partial \Om} \overline{\Delta} f \partial_\nu f
& = 
\int_{\partial \Om} \frac{1}{r} |\nabla_\partial f|^2
+ 
\int_{\partial \Om}  \langle \nabla_\partial f, \nabla_\nu \nabla_\partial f \rangle .
\end{split}
\label{second_by_parts_positivity}
\end{align}
Choose Fermi coordinates as explained at the beginning
 of this section. Then $\nabla^\mu \nabla_\mu f = 0$
implies $\nabla^\mu \nabla_\mu \nabla_\si f = 0$,
 and $\nabla^\si f \nabla^\mu \nabla_\mu \nabla_\si f =0$.
Integrating over $\Om_\varepsilon$ and integrating by parts,
\begin{align}
\begin{split}
0 = \int_{\Om_\varepsilon} 
\nabla^\si f \nabla^\mu \nabla_\mu \nabla_\si f
& = - \int_{\Om_\varepsilon} \nabla^\mu \nabla^\si f \nabla_\mu \nabla_\si f 
+ \int_{\partial \Om_\varepsilon} \nabla^\si f \nabla_V \nabla_\si f ,
\end{split}
\nonumber
\end{align}
so
\begin{align}
\begin{split}
\int_{\Om_\varepsilon} \nabla^\mu \nabla^\si f \nabla_\mu \nabla_\si f  
& = \int_{\partial \Om_\varepsilon^1}  \nabla^\si f \nabla_V \nabla_\si f
+ \int_{\partial \Om_\varepsilon^2} \nabla^\si f \nabla_V \nabla_\si f,
\end{split}
\label{split_partial_1_2}
\end{align}
where $\partial \Om_\varepsilon^1 = 
\partial \Om \backslash (\partial \Om_\varepsilon \cap \partial \Om)$,
$\partial \Om_\varepsilon^2 = \partial \Om_\varepsilon \backslash \partial \Om_\varepsilon^1$ 
(see figure \ref{figure_Om_e}),
and $\nabla_V$ is covariant differentiation in the normal direction, with $V$ the unit
outer normal. On $\partial \Om_\varepsilon^1$, $\nabla_V = -\nabla_3$. Compute
\begin{align}
\begin{split}
\nabla_3 \nabla_\si f & = 
\partial_3 \partial_\si f - \Ga_{3\si}^j \partial_j f,
\end{split}
\nonumber
\end{align}
where we used that any Christoffel symbol with two indices $3$ vanishes in 
Fermi coordinates. Let $\overline{\nabla}$ be the covariant derivative on
$\partial \Om$. Then, since $\nabla_i f = \partial_i f = \overline{\nabla}_i f$,
\begin{align}
\begin{split}
\nabla_3 \nabla_i f & =  
\nabla_3 \overline{\nabla}_i f. 
\end{split}
\nonumber
\end{align}
Hence,
\begin{align}
\begin{split}
\int_{\partial \Om_\varepsilon^1}  \nabla^\si f \nabla_V \nabla_\si f
& = 
- \int_{\partial \Om_\varepsilon^1}  \overline{\nabla}^i f \nabla_3 \overline{\nabla}_i f 
- \int_{\partial \Om_\varepsilon^1}  \nabla^3 f \nabla_3 \nabla_3 f \\
\end{split}
\label{cov_diff_split_i_n}
\end{align}
Using (\ref{cov_diff_split_i_n}) into (\ref{split_partial_1_2}) gives
\begin{align}
\begin{split}
\int_{\Om_\varepsilon} \nabla^\mu \nabla^\si f \nabla_\mu \nabla_\si f  
& = 
- \int_{\partial \Om_\varepsilon^1}  \overline{\nabla}^i f \nabla_3 \overline{\nabla}_i f 
- \int_{\partial \Om_\varepsilon^1}  \nabla^3 f \nabla_3 \nabla_3 f
+ \int_{\partial \Om_\varepsilon^2} \nabla^\si f \nabla_V \nabla_\si f,
\end{split}
\label{intermediate_partial_1}
\end{align}
Taking $\nabla_3$ of $\nabla^\mu \nabla_\mu f =0$, commuting the
covariant derivatives, multiplying by $\nabla^3 f$ and 
integrating by parts,
\begin{align}
\begin{split}
0 = \int_{\Om_\varepsilon} 
\nabla^3 f \nabla^\mu \nabla_\mu \nabla_3 f
& = - \int_{\Om_\varepsilon} \nabla^\mu \nabla^3 f \nabla_\mu \nabla_3 f 
+ \int_{\partial \Om_\varepsilon} \nabla^3 f \nabla_V \nabla_3 f ,
\end{split}
\nonumber
\end{align}
so, since $\nabla_V = - \nabla_3$ on $\partial \Om_\varepsilon^1$,
\begin{align}
\begin{split}
 \int_{\partial \Om_\varepsilon^1}  \nabla^3 f \nabla_3 \nabla_3 f
& = 
 -\int_{\Om_\varepsilon} \nabla^\mu \nabla^3 f \nabla_\mu \nabla_3 f  
+ \int_{\partial \Om_\varepsilon^2} \nabla^3 f \nabla_V \nabla_3 f 
\end{split}
\nonumber
\end{align}
Using the above expression for 
$\int_{\partial \Om_\varepsilon^1}  \nabla^3 f \nabla_3 \nabla_3 f$ into
(\ref{intermediate_partial_1}),
\begin{align}
\begin{split}
\int_{\Om_\varepsilon} \nabla^\mu \nabla^\si f \nabla_\mu \nabla_\si f  
& = 
- \int_{\partial \Om_\varepsilon^1}  \overline{\nabla}^i f \nabla_3 \overline{\nabla}_i f
+\int_{\Om_\varepsilon} \nabla^\mu \nabla^3 f \nabla_\mu \nabla_3 f  
\\
& - \int_{\partial \Om_\varepsilon^2} \nabla^3 f \nabla_V \nabla_3 f 
+ \int_{\partial \Om_\varepsilon^2} \nabla^\si f \nabla_V \nabla_\si f,
\end{split}
\nonumber
\end{align}
so that
\begin{align}
\begin{split}
 \int_{\partial \Om_\varepsilon^1}  \overline{\nabla}^i f \nabla_V \overline{\nabla}_i f
 = 
- \int_{\partial \Om_\varepsilon^1}  \overline{\nabla}^i f \nabla_3 \overline{\nabla}_i f
& = 
\int_{\Om_\varepsilon} \nabla^\mu \nabla^\si f \nabla_\mu \nabla_\si f  
- \int_{\Om_\varepsilon} \nabla^\mu \nabla^3 f \nabla_\mu \nabla_3 f  
\\
& + \int_{\partial \Om_\varepsilon^2} \nabla^3 f \nabla_V \nabla_3 f 
- \int_{\partial \Om_\varepsilon^2} \nabla^\si f \nabla_V \nabla_\si f.
\end{split}
\label{estimate_bry_interior_epsilon}
\end{align}
The first two terms on the right hand side combine to give 
\begin{gather}
\int_{\Om_\varepsilon} \left( \nabla^\mu \nabla^\si f \nabla_\mu \nabla_\si f  
- \nabla^\mu \nabla^3 f \nabla_\mu \nabla_3 f  \right)
= \int_{\Om_\varepsilon} \left( |\nabla\nabla f|^2 - |\nabla \nabla_V  f |^2\right) \geq 0.
\nonumber
\end{gather}
For the integrals along $\partial \Om_\varepsilon^2$, write
$\partial \Om_\varepsilon^2 = \partial \Om_\varepsilon^{21} \cup \partial \Om_\varepsilon^{22}$,
where
\begin{gather}
\partial \Om_\varepsilon^{21} = \partial \Om_\varepsilon^{2} \cap \partial B_\varepsilon(0),
\nonumber
\end{gather}
and
\begin{gather}
\partial \Om_\varepsilon^{22} = \partial \Om_\varepsilon^{2} \backslash
\partial \Om_\varepsilon^{21}.
\nonumber
\end{gather}
Then
\begin{align}
\begin{split}
\left|
 \int_{\partial \Om_\varepsilon^{21}} \nabla^3 f \nabla_V \nabla_3 f 
- \int_{\partial \Om_\varepsilon^{21}} \nabla^\si f \nabla_V \nabla_\si f
\right| \leq & 
 \int_{\partial \Om_\varepsilon^{21}} \left| \nabla^3 f \nabla_V \nabla_3 f \right|
+ \int_{\partial \Om_\varepsilon^{21}} \left| \nabla^\si f \nabla_V \nabla_\si f \right|
\\
\leq & 
 \int_{\partial B_\varepsilon(0)} \left| \nabla^3 f \nabla_V \nabla_3 f \right|
+ \int_{\partial B_\varepsilon(0)} \left| \nabla^\si f \nabla_V \nabla_\si f \right| \\
\leq & 
C \int_{\partial B_\varepsilon(0)} \p f \p^2_{2,\partial B_\varepsilon(0)} \\
\leq & C \int_{\partial B_\varepsilon(0)} \p f \p^2_{\frac{5}{2}} \\
\leq & C \p f \p^2_{\frac{5}{2}} \varepsilon^2,
\end{split}
\label{error_ball}
\end{align} 
where $\p \cdot \p_{2,\partial B_\varepsilon(0)}$ is the Sobolev norm on $\partial B_\varepsilon(0)$,
and on the next-to-the-last step we used (\ref{restriction}).
Similarly,
\begin{align}
\begin{split}
\left|
 \int_{\partial \Om_\varepsilon^{22}} \nabla^3 f \nabla_V \nabla_3 f 
- \int_{\partial \Om_\varepsilon^{22}} \nabla^\si f \nabla_V \nabla_\si f
\right| \leq & 
C \p f \p^2_{\frac{5}{2}} \varepsilon.
\end{split}
\label{error_cone}
\end{align}
From (\ref{error_ball}) and (\ref{error_cone}) we see that the integrals over $\partial \Om_\varepsilon^2$
vanish in the limit $\varepsilon \rar 0^+$. Hence,
\begin{gather}
\int_{\partial \Om} \langle \nabla_\partial f, \nabla_\nu \nabla_\partial f \rangle 
= \lim_{\varepsilon \rar 0^+} \int_{\partial \Om_\varepsilon^1}  \overline{\nabla}^i f \nabla_V \overline{\nabla}_i f
  \geq 0,
\label{positivity_limit_epsilon}
\end{gather}
where we recall that $\nu$ is the outer unit normal to $\partial \Om$.
Combining (\ref{positivity_limit_epsilon})
 with (\ref{second_by_parts_positivity}) yields the result.
\end{proof}
As a consequence of lemma \ref{positivity_lemma}, we have
$E(t) \geq 0$, and
\begin{gather}
\p \dot{f} \p_{0, \partial}^2 \leq E.
\label{f_dot_partial_E}
\end{gather}
Thus, (\ref{dot_E_CS}) gives
\begin{gather}
\dot{E} \leq \frac{1}{2} \p G \p_{0,\partial}^2 + \frac{1}{2} E,
\nonumber
\end{gather}
or,
\begin{gather}
E(t)  \leq E(0) + \frac{1}{2} \int_0^t  \p G(\tau) \p_{0,\partial}^2 \, d\tau 
+ \frac{1}{2} \int_0^t  E(\tau) \, d\tau,
\nonumber
\end{gather}
which gives, after iteration, 
\begin{gather}
E(t) \leq \left(  E(0) + \frac{1}{2} \int_0^t  \p G(\tau) \p_{0,\partial}^2 \, d\tau 
\right) e^{\frac{1}{2} t}.
\label{simple_energy_estimate_linear}
\end{gather}
From
(\ref{harmonic_ext_linear}) and standard elliptic theory we have
\begin{gather}
\p f \p_s \leq C \p f \p_{s-\frac{1}{2},\partial}.
\label{elliptic_estimate_Laplace_eq}
\end{gather}
And invoking elliptic theory once more,  (\ref{elliptic_estimate_Laplace_eq}) then gives 
the following estimate:
\begin{align}
\begin{split}
\left| \int_{\partial \Om} f \overline{\Delta} \partial_\nu f \right |
& = \left| \int_{\partial \Om}  \overline{\Delta}f \partial_\nu f \right | \\
& \leq \p \overline{\Delta}f \p_{-\frac{1}{2},\partial} \p \partial_\nu f \p_{\frac{1}{2},\partial} 
\\
& \leq C\p f \p_{\frac{3}{2}, \partial } \p f \p_2 \\
& \leq C \p f \p_{\frac{3}{2}, \partial}^2.
\end{split}
\label{energy_less_3_2}
\end{align}
Thus, (\ref{energy_less_3_2}) and the definition of $E$, i.e. (\ref{def_E}), imply
\begin{gather}
E \leq C \p f \p_{\frac{3}{2}, \partial }^2 + C \p \dot{f} \p_{0,\partial}^2.
\label{energy_less_norm}
\end{gather}
Combining (\ref{energy_less_norm}) at time zero with (\ref{simple_energy_estimate_linear})
produces
\begin{gather}
E(t) \leq C \left(  
\p f(0) \p_{\frac{3}{2}, \partial }^2 + C \p \dot{f}(0) \p_{0,\partial}^2
+ \frac{1}{2} \int_0^t  \p G(\tau) \p_{0,\partial}^2 \, d\tau 
\right) e^{\frac{1}{2} t}.
\label{simple_energy_estimate_linear_2}
\end{gather}
Next, we show that
\begin{lemma}
\begin{gather}
-\int_{\partial \Om} f \overline{\Delta} \partial_\nu f =
 - \int_{\partial \Om}
\overline{\Delta} f \partial_\nu f \geq C \p f \p^2_{\frac{3}{2}, \partial } 
-C \p f \p_{0,\partial}^2.
\nonumber
\end{gather}
\label{lemma_3_2_norm}
\end{lemma}
\begin{rema}
The presence of the negative term $-\p f \p_{0,\partial}^2$ is necessary as
$\int_{\partial \Om} \overline{\Delta} f \partial_\nu f$ is zero on constant functions.
\end{rema}
\begin{proof}
The equality follows by integration by parts. In light of
(\ref{second_by_parts_positivity}), it suffices to obtain the inequality for
\begin{align}
\begin{split}
\int_{\partial \Om}  \langle \nabla_\partial f, \nabla_\nu \nabla_\partial f \rangle ,
\end{split}
\nonumber
\end{align}
which, as in lemma
\ref{positivity_lemma}, 
along $\partial \Om_\varepsilon^1$ corresponds to the term
\begin{gather}
- \int_{\partial \Om_\varepsilon^1}  \overline{\nabla}^i f \nabla_3 \overline{\nabla}_i f.
\nonumber
\end{gather}
Proceed as in lemma \ref{positivity_lemma}
 until (\ref{estimate_bry_interior_epsilon}), and consider its 
first two terms on the right hand side. Their integrand  gives
\begin{gather}
 \nabla^\mu \nabla^\si f \nabla_\mu \nabla_\si f  
- \nabla^\mu \nabla^3 f \nabla_\mu \nabla_3 f  
=  \nabla^\mu \nabla^i f \nabla_\mu \nabla_i f.  
\nonumber
\end{gather}
Letting $\overline{\nabla}$ denote the covariant derivative along
$\partial \Om_r \cap \Om_\varepsilon$ (which correspond to the level sets
$x^3 = \text{constant}$), noticing that $\overline{\nabla}_i f = \partial_i f
= \nabla_i f$ and that, therefore, $\overline{\nabla}_i$ gives a well-defined
operator on $\Om_\varepsilon$, we can write the above as
\begin{gather}
 \nabla^\mu \nabla^\si f \nabla_\mu \nabla_\si f  
- \nabla^\mu \nabla^3 f \nabla_\mu \nabla_3 f  
=  | \nabla \overline{\nabla} f|^2,  
\nonumber
\end{gather}
thus
(\ref{estimate_bry_interior_epsilon}) gives, 
\begin{align}
\begin{split}
- \int_{\partial \Om_\varepsilon^1}  \overline{\nabla}^i f \nabla_3 \overline{\nabla}_i f
& = 
\int_{\Om_\varepsilon} |\nabla \overline{\nabla} f|^2 
+ \int_{\partial \Om_\varepsilon^2} \nabla^3 f \nabla_V \nabla_3 f 
- \int_{\partial \Om_\varepsilon^2} \nabla^\si f \nabla_V \nabla_\si f.
\end{split}
\label{initial_integral_lemma_3_2}
\end{align}
But
\begin{align}
\begin{split}
\int_{\Om_\varepsilon} |\nabla \overline{\nabla} f|^2 
\geq 
\p \overline{\nabla} f \p_1^2 - C\p  f\p_1^2 
 \geq 
\p \overline{\nabla} f \p_1^2 - C\p  f\p_{\frac{1}{2}, \partial_\varepsilon}^2  
 \geq 
C \p \overline{\nabla} f \p_{\frac{1}{2},\partial_\varepsilon}^2 
- C\p  f\p_{\frac{1}{2}, \partial_\varepsilon}^2 ,
\end{split}
\label{second_der_f_lemma_3_2}
\end{align}
where in the last step we used (\ref{restriction}), in the next-to-the-last,
(\ref{elliptic_estimate_Laplace_eq}), and
$\p  \cdot \p_{s, \partial_\varepsilon}$ is the Sobolev norm on the 
boundary $\partial \Om_\varepsilon$. Applying the interpolation inequality 
(\ref{interpolation}) with $s_1 = 0$, $s_2 = \frac{1}{2}$, and $s_3 = 1$,
\begin{align}
\begin{split}
\p  f\p_{\frac{1}{2}, \partial_\varepsilon}^2 \leq & 
\frac{C}{\gamma} \p f \p_{0,\partial_\varepsilon}^2 
+ \gamma \p f \p_{1,\partial_\varepsilon}^2 ,
\end{split}
\label{interpolation_lemma_3_2}
\end{align}
where the Cauchy inequality with $\gamma$, (\ref{Cauchy_epsilon}), 
has been employed.
Using (\ref{interpolation_lemma_3_2}) in (\ref{second_der_f_lemma_3_2}) and recalling
(\ref{fractional_norm}),
\begin{align}
\begin{split}
\int_{\Om_\varepsilon} |\nabla \overline{\nabla} f|^2 
 \geq &
 C \p \overline{\nabla} f \p_{\frac{1}{2},\partial_\varepsilon}^2
 - C \gamma \p f \p_{1,\partial_\varepsilon}^2
- \frac{C}{\gamma} \p f \p_{0,\partial_\varepsilon}^2  \\
= & 
C \left[ \overline{\nabla} f \right]_{\frac{1}{2}, \partial_\varepsilon}^2 + 
 C \p f \p_{1,\partial_\varepsilon}^2
- C \gamma \p f \p_{1,\partial_\varepsilon}^2
- \frac{C}{\gamma} \p f \p_{0,\partial_\varepsilon}^2 \\
\geq & 
C \left[ \overline{\nabla} f \right]_{\frac{1}{2}, \partial_\varepsilon}^2 + 
 C \p f \p_{1,\partial_\varepsilon}^2
- \frac{C}{\gamma} \p f \p_{0,\partial_\varepsilon}^2,
\end{split}
\label{second_der_f_lemma_3_2_2}
\end{align}
where the last step follows by choosing $\gamma$ sufficiently small.
Since $\overline{\nabla}$ is differentiation along $\partial \Om_\varepsilon$, 
recalling (\ref{fractional_norm}) once more, we see that
\begin{gather}
C \left[ \overline{\nabla} f \right]_{\frac{1}{2}, \partial_\varepsilon}^2 + 
 C \p f \p_{1,\partial_\varepsilon}^2
 \geq C \p f \p_{\frac{3}{2}, \partial_\varepsilon}^2,
 \nonumber
 \end{gather}
so that (\ref{second_der_f_lemma_3_2_2}) gives
\begin{align}
\begin{split}
\int_{\Om_\varepsilon} |\nabla \overline{\nabla} f|^2 
 \geq &  C\p f \p^2_{\frac{3}{2}, \partial_\varepsilon}
 - C \p f \p^2_{0,\partial_\varepsilon},
  \end{split}
\nonumber
\end{align}
and therefore (\ref{initial_integral_lemma_3_2}) implies
\begin{align}
\begin{split}
- \int_{\partial \Om_\varepsilon^1}  \overline{\nabla}^i f \nabla_3 \overline{\nabla}_i f
& \geq 
C\p f \p^2_{\frac{3}{2}, \partial_\varepsilon}
 - C \p f \p^2_{0,\partial_\varepsilon}
+ \int_{\partial \Om_\varepsilon^2} \nabla^3 f \nabla_V \nabla_3 f 
- \int_{\partial \Om_\varepsilon^2} \nabla^\si f \nabla_V \nabla_\si f.
\end{split}
\label{last_step_lemma_3_2}
\end{align}
To finish the proof, split the first two terms on the right hand
of (\ref{last_step_lemma_3_2}) in integrals along $\partial \Om_\varepsilon^{21}$ and
$\partial \Om_\varepsilon^{22}$. Arguing as in lemma \ref{positivity_lemma}, 
all integrals on $\partial \Om_\varepsilon^2$ vanish in the limit $\varepsilon\rar 0^+$,
which gives the result.
\end{proof}

As a consequence of lemma \ref{lemma_3_2_norm} and the definition
\ref{def_E}, we have
\begin{align}
\begin{split}
\p f \p^2_{\frac{3}{2},\partial} & \leq \frac{C}{\kappa} E 
+ C \p f \p_{0,\partial}^2.
\end{split}
\label{estimate_f_3_2_Energy_1}
\end{align}
Using the fundamental theorem of calculus and the Cauchy-Schwarz inequality,
\begin{align}
\begin{split}
\p f \p_{0,\partial}^2 & \leq 
 C \p f(0) \p_{0,\partial}^2 + C \left( \int_0^t \p \dot{f} \p_{0,\partial} \right)^2
\leq  C \p f(0) \p_{0,\partial}^2 +  C t \int_0^t \p \dot{f} \p_{0, \partial}^2.
\end{split}
\label{estimate_f_3_2_Energy_2}
\end{align}
In the above, we used  Jensen's inequality,
\begin{gather}
h\left( \dashint f \right)  \leq \dashint h(f),
\nonumber
\end{gather}
where $h$ is a convex function and $\dashint$ the average over the domain of integration,
to estimate
\begin{align}
\begin{split}
\left( \int_0^t \p \dot{f} \p_{0,\partial} \right)^2 
& = 
t^2 \left( \dashint_0^t \p \dot{f} \p_{0,\partial} \right)^2 
\leq t \int_0^t \p \dot{f} \p_{0, \partial}^2.
\end{split}
\nonumber
\end{align}
Thus, (\ref{f_dot_partial_E}),
 (\ref{estimate_f_3_2_Energy_1}), and (\ref{estimate_f_3_2_Energy_2})  give
\begin{align}
\begin{split}
\p f \p^2_{\frac{3}{2},\partial} & \leq \frac{C}{\kappa} E 
+ C \p f(0) \p_{0,\partial}^2 +  C t \int_0^t E.
\end{split}
\label{estimate_f_3_2_Energy_bound}
\end{align}
Inequalities (\ref{f_dot_partial_E}) and (\ref{estimate_f_3_2_Energy_bound})
combined with (\ref{simple_energy_estimate_linear_2}) give bounds 
for $\p \dot{f} \p_{0,\partial}$ and $\p f \p^2_{\frac{3}{2},\partial}$.
More precisely, we have
\begin{gather}
\p \dot{f} \p_{0, \partial}^2  \leq C \left(  
\p f(0) \p_{\frac{3}{2}, \partial }^2 +  \p \dot{f}(0) \p_{0,\partial}^2
+ \frac{1}{2} \int_0^t  \p G(\tau) \p_{0,\partial}^2 \, d\tau 
\right) e^{\frac{1}{2} t},
\label{bound_f_dot}
\end{gather}
and
\begin{align}
\begin{split}
\p f \p^2_{\frac{3}{2},\partial} & \leq 
 C \p f(0) \p_{0,\partial}^2  + \frac{C}{\kappa}
\left(  
\p f(0) \p_{\frac{3}{2}, \partial }^2 +  \p \dot{f}(0) \p_{0,\partial}^2
+ \frac{1}{2} \int_0^t  \p G(\tau) \p_{0,\partial}^2 \, d\tau 
\right) e^{\frac{1}{2} t} \\
& + C t^2 e^{\frac{1}{2}T} \left(
\p f(0) \p_{\frac{3}{2}, \partial }^2 +  \p \dot{f}(0) \p_{0,\partial}^2
+ \int_0^t  \p G(\tau) \p_{0,\partial}^2 \, d\tau 
\right),
\end{split}
\label{bound_f_3_2_}
\end{align}
where we used $e^{\frac{1}{2} t} \leq e^{\frac{1}{2} T}$, and 
\begin{gather}
\int_0^t  \int_0^\tau \p G(\si) \p_{0,\partial}^2 \, d\si d\tau 
\leq  
\int_0^t  \int_0^t \p G(\si) \p_{0,\partial}^2 \, d\si d\tau 
= t \int_0^t  \p G(\si) \p_{0,\partial}^2 \, d\si. 
\nonumber
\end{gather}

\section{Proofs.}
We are now ready to proof proposition \ref{prop_extension} and
theorem \ref{main_theorem}.
Let $X^{3,+}_T(\partial \Om)$ be the subspace of $X^3_T(\partial \Om)$
 consisting of functions
such that $f(0) = \dot{f}(0) = 0$, and
$X^{3,-}_T(\partial \Om)$ be the subspace of $X^3_T(\partial \Om)$ consisting of functions
such that $f(T) = \dot{f}(T) = 0$.
 Recall that
$X_T^3(\partial \Om) \subset X_T^{\frac{3}{2}}(\partial \Om)$.

\begin{lemma}
$X^{3,\pm}_T(\partial \Om)$ is dense in $L^2(T)$.
\end{lemma}
\begin{proof}
This is very similar to the proof that compactly supported functions are dense
in Sobolev spaces, using mollifiers.
\end{proof}

\noindent \emph{Proof of proposition \ref{prop_extension}-(i):} 
From
(\ref{bound_f_dot}) and (\ref{bound_f_3_2_}) with $G = 0$, $f(0) = 0$, and $\dot{f}(0)=0$,
it follows that $\cL$ is injective.

Let $(G, f_0, f_1) \in \cH$ and $f = \cL^{-1}\big( (G, f_0, f_1) \big)$. Then
(\ref{bound_f_3_2_}) gives
\begin{gather}
\p f \p_{\frac{3}{2}, \partial} \leq C(T,\kappa) 
\left(
  \p f_0 \p_{\frac{3}{2}, \partial} 
+\p f_1 \p_{0,\partial} + \sqrt{ \int_0^T \p G(\tau) \p_{0,\partial}^2 \, d\tau }
\right),
\nonumber
\end{gather}
where $C(T, \kappa)$ is a constant that depends on $T$ and $\kappa$, and 
we have used that  
\begin{gather}
\int_0^t \p G(\tau) \p_{0,\partial}^2 \ d\tau 
\leq 
\int_0^T \p G(\tau) \p_{0,\partial}^2 \ d\tau, \, t \in [0,T],
\nonumber
\end{gather}
and that the exponential is an increasing function.
But
\begin{gather}
\int_0^T \p G(\tau) \p_{0,\partial}^2 \, d\tau
= \int_0^T \int_\Om |G(\tau,x)|^2 \, dx\, d\tau \leq C \p G \p_{L^2(T)}^2,
\nonumber
\end{gather}
so that 
\begin{gather}
\p f \p_{\frac{3}{2}, \partial} \leq C(T,\kappa) 
\left(
  \p f_0 \p_{\frac{3}{2}, \partial} 
+\p f_1 \p_{0,\partial} + \p G \p_{L^2(T)}
\right).
\nonumber
\end{gather}
 Similarly, from (\ref{bound_f_dot}) we also conclude,
\begin{gather}
\p \dot{f} \p_{0, \partial} \leq C(T,\kappa) 
\left(
  \p f_0 \p_{\frac{3}{2}, \partial} 
+\p f_1 \p_{0,\partial} + \p G \p_{L^2(T)}
\right).
\nonumber
\end{gather} 
These last two inequalities imply
\begin{gather}
\p f \p_{X^\frac{3}{2}_T(\partial \Om)} \leq C(T,\kappa) 
\left(
  \p f_0 \p_{\frac{3}{2}, \partial} 
+\p f_1 \p_{0,\partial} + \p G \p_{L^2(T)}
\right).
\label{continuity_L_inverse}
\end{gather} 
As the right hand side of this last inequality is the norm of 
$(G,f_0,f_1)$ in the topology of $\cH$, we conclude that $\cL^{-1}$ is a continuous map.
\qed \\

\noindent \emph{Proof of proposition \ref{prop_extension}-(ii):} 
Denoting by $\overline{\cR}$ the closure
of $\cR$ in $\cH$, $\cL^{-1}$ extends, by continuity, to a continuous linear map
\begin{gather}
\overline{\cL^{-1}}: \overline{\cR} \rar X^\frac{3}{2}_T(\partial \Om).
\nonumber
\end{gather}
In order to show that the closure of $\cR$ in $\cH$ is the whole of $\cH$, i.e., 
\begin{gather}
\overline{\cR} = \cH,
\nonumber
\end{gather}
we will prove that if $v \perp \overline{\cR}$, then 
$ v = 0$. As $\cR$ is dense in its closure, it suffices to show that 
\begin{gather}
\text{if } (v,w)_\cH = 0 \text{ for all } w = (G,f_0, f_1) \in \cR,
\text{ then } v = 0,
\nonumber
\end{gather}
where $(\cdot, \cdot, \cdot)_\cH$ is the inner product
on $\cH \equiv L^2(T) \times H^\frac{3}{2}(\partial \Om) \times 
H^0(\partial \Om)$. Write $v = (H, v_0, v_1)$, $w = w = (G,f_0, f_1)$, and suppose that
\begin{gather}
(H,G)_{L^2(T)} + (v_0, f_0)_{\frac{3}{2}, \partial} + (v_1, f_1)_{0,\partial} = 0
\text{ for all } (G,f_0,f_1) \in \cR,
\nonumber
\end{gather}
where $(\cdot, \cdot)_{L^2(T)}$, $(\cdot, \cdot)_{\frac{3}{2},\partial}$, 
and $(\cdot, \cdot)_{0,\partial}$ are the inner products in
$ L^2(T)$, $H^\frac{3}{2}(\partial \Om)$, and $H^0(\partial \Om)$, respectively.
By definition of $\cR$ and the fact that $\cL$ is injective, the above means
\begin{gather}
(H,\ddot{f} - \kappa \overline{\Delta} \partial_\nu f)_{L^2(T)} 
+ (v_0, f(0))_{\frac{3}{2}, \partial} + (v_1, \dot{f}(0))_{0,\partial} = 0
\text{ for all } f \in X^3_T(\partial \Om).
\label{inner_product_zero_1}
\end{gather}
Assume first that $f \in X^{3,+}_T(\partial \Om)$. In this case 
(\ref{inner_product_zero_1}) becomes
\begin{gather}
(H, \ddot{f} - \kappa \overline{\Delta} \partial_\nu f)_{L^2(T)} = 0.
\label{inner_product_zero_2}
\end{gather}
Suppose, further, that $H \in X^{3,-}_T(\partial \Om)$. Then
\begin{align}
\begin{split}
(H, \ddot{f} - \kappa \overline{\Delta} \partial_\nu f)_{L^2(T)} & = 
\int_{[0,T]\times \partial \Om} H ( \ddot{f} - \kappa \overline{\Delta} \partial_\nu f )
 \\
& = 
\int_{ \partial \Om} \int_0^T H  \ddot{f} \, dt dx - 
\kappa \int_{ \partial \Om} \int_0^T H  \overline{\Delta} \partial_\nu f \, dt dx 
\end{split}
\label{integration_parts_L2_T}
\end{align}
Integrating by parts in the time variable  the first term,
\begin{align}
\begin{split}
\int_{ \partial \Om} \int_0^T H  \ddot{f} \, dt dx
 & = 
\int_{ \partial \Om}  \int_0^T \ddot{H} f  \, dt  dx,
\end{split}
\label{integration_parts_time_H_f}
\end{align}
where we used that $f \in X^{3,+}_T(\partial \Om)$ and
$H \in X^{3,-}_T(\partial \Om)$.
For the second term in (\ref{integration_parts_L2_T}),
switch the order of integration and integrate by parts the Laplacian term to get
\begin{gather}
\int_{ \partial \Om} \int_0^T H  \overline{\Delta} \partial_\nu f \, dt dx 
= 
 \int_0^T \int_{ \partial \Om}  \overline{\Delta} H \partial_\nu f \,  dx dt. 
\nonumber
\end{gather}
Next, consider the harmonic extensions of $f$ and $H$ to the whole of $\Om$, which 
we still denote by $f$, and $H$, respectively.
 Letting $\varphi(r) = r^2$,  using  Green's
identity, and arguing as in section \ref{section_basic_energy},
\begin{align}
\begin{split}
\int_\Om \left( \varphi \overline{\Delta} H \Delta f - f \Delta( \varphi \overline{\Delta} H )
\right) 
& = \int_{\partial \Om} \left( \overline{\Delta} H \partial_\nu f  - f \partial_\nu (\varphi
\overline{\Delta} H ) \right) 
= \int_{\partial \Om} \left( \overline{\Delta} H \partial_\nu f  - 
f \overline{\Delta} \partial_\nu H \right),
\end{split}
\label{integration_parts_spatial_H_f}
\end{align}
where we used (\ref{induced_Laplacian_spherical}) so that, on $\partial \Om$,
\begin{gather}
\partial_\nu (\varphi\overline{\Delta} H ) = \partial_r(r^2 \overline{\Delta} H )
= \partial_r \Delta_{S^2} H = \Delta_{S^2} \partial_r H =
\overline{\Delta} \partial_\nu H.
\nonumber
\end{gather} 
Using also (\ref{Laplacian_spherical}),
\begin{align}
\begin{split}
\Delta (\varphi \overline{\Delta} H ) & = 
\Delta_{S^2} \partial_r^2 H
+ \Delta_{S^2}\left( \frac{2}{r} \partial_r H \right)
+ \Delta_{S^2} \left (\frac{1}{r^2} \Delta_{S^2} H \right) 
= \Delta_{S^2} \Delta H = 0.
\end{split}
\nonumber
\end{align}
And as $\Delta f = 0$, (\ref{integration_parts_spatial_H_f}) gives
\begin{gather}
\int_{\partial \Om} \overline{\Delta} H \partial_\nu f  = 
\int_{\partial \Om} f \overline{\Delta} \partial_\nu H.
\label{integration_parts_spatial_H_f_result}
\end{gather}
Using (\ref{integration_parts_time_H_f}) and (\ref{integration_parts_spatial_H_f_result}) into
(\ref{integration_parts_L2_T}) yields,
\begin{gather}
(H, \ddot{f} - \overline{\Delta} \partial_\nu f )_{L^2(T)} 
= (\ddot{H} - \overline{\Delta} \partial_\nu H, f )_{L^2(T)}.
\nonumber
\end{gather}
Since this holds for all $ f \in X^{3,+}_T(\partial \Om)$ we conclude, from 
(\ref{inner_product_zero_1}) 
and the density 
of $X^{3,+}_T(\partial \Om)$ in $L^2(T)$, that 
$(\ddot{H} - \overline{\Delta} \partial_\nu H, f )_{L^2(T)} = 0$ for all $f \in L^2(T)$, 
and thus
\begin{gather}
\ddot{H} - \overline{\Delta} \partial_\nu H = 0.
\nonumber
\end{gather}
Recall that $H(T) = \dot{H}(T) = 0$ because $H \in X^{3,-}_T(\partial \Om)$.
But solutions to the above equation, with these boundary conditions, are unique;
this follows using the same argument used to show that $\cL^{-1}$
is injective, i.e., using energy estimates, except 
 with the roles of $0$ and $T$ reversed.
Thus $H = 0$. Since $X^{3,-}_T(\partial \Om)$ is dense in $L^2(T)$, we conclude from
the continuity of the inner product that if $(v,w)_\cH = 0$ for all $w \in \cR$, then
$v = (0, v_0, v_1)$. 
(\ref{inner_product_zero_1}) therefore reduces to
\begin{gather}
(v_0, f(0))_{\frac{3}{2}, \partial} + (v_1, \dot{f}(0))_{0,\partial} = 0
\text{ for all } f \in X^3_T(\partial \Om).
\nonumber
\end{gather}
But if $f$ is an arbitrary element of $X^3_T(\partial \Om)$, then $f(0)$ and
$\dot{f}(0)$ are arbitrary elements of $H^3(\partial \Om)$ and $H^\frac{3}{2}(\partial \Om)$,
respectively. Since these last two spaces are dense in, respectively,
 $H^\frac{3}{2}(\partial \Om)$ and $H^0(\partial \Om)$, we conclude that
 \begin{gather}
(v_0, f_0)_{\frac{3}{2}, \partial} + (v_1, f_1)_{0,\partial} = 0
\text{ for all } (f_0, f_1) \in H^\frac{3}{2}(\partial \Om) \times H^0(\partial \Om),
\nonumber
\end{gather}
and thus $v = 0$, as desired. \qed

\noindent \emph{Proof of proposition \ref{prop_extension}-(iii):}
We already know that $\overline{\cL^{-1}}$ is defined on the 
whole of $\cH$, and any $u \in \cH$ is the limit of a sequence in $\cR$. 
Let $ y = \overline{\cL^{-1}}(u)$, and take a sequence $\{ u_\ell \} \subset \cR$ 
converging in $\cH$ to $u$. Then 
\begin{gather}
\overline{\cL^{-1}}(u_\ell ) = \cL^{-1}(u_\ell) \equiv y_\ell.
\nonumber
\end{gather}
By construction, $\cL^{-1}$ is the inverse of $\cL$ defined on $X^3_T(\partial \Om)$, thus
$y_\ell \in X^3_T(\partial \Om)$. As we showed that $\cL^{-1}$ is a continuous map
(see (\ref{continuity_L_inverse})), we have
\begin{align}
\begin{split}
\p y - y_\ell \p_{X^\frac{3}{2}_T(\partial \Om)} 
= 
\p \overline{\cL^{-1}}(u) - \cL^{-1}(u_\ell) \p_{X^\frac{3}{2}_T(\partial \Om)} 
=
\p \overline{\cL^{-1}}(u - u_\ell) \p_{X^\frac{3}{2}_T(\partial \Om)} 
\leq C \p u - u_\ell \p_\cH,
\end{split}
\nonumber
\end{align}
implying that $y_\ell \rar y$ in $X^\frac{3}{2}_T(\partial \Om)$, and thus
$y \in \overline{X^3_T(\partial \Om)\,}{}^{X^\frac{3}{2}_T(\partial \Om)}$.

Injectivity also follows from the continuity of $\overline{\cL^{-1}}$.
Say $\overline{\cL^{-1}}(u) = 0$, and let $\{ u_\ell \}$ be as above.
Then
\begin{gather}
0 = \overline{\cL^{-1}}(u) = \overline{\cL^{-1}}(\lim u_\ell)
= \lim \overline{\cL^{-1}}(u_\ell),
\nonumber
\end{gather}
which implies $u_\ell = 0$ for every $\ell$ since 
$\overline{\cL^{-1}} = \cL^{-1}$ on $\cR$ and $\cL^{-1}$ is injective.
Thus $u = 0$. \qed \\

\noindent \emph{Proof of theorem  \ref{main_theorem}:} 
The existence and uniqueness of a weak solution follows at once from 
proposition \ref{prop_extension} , upon noticing that if
$f \in X_T^\frac{3}{2}(\partial \Om)$, then, by elliptic theory,
 its harmonic extension
is in $X_T^2( \Om)$. \qed

\end{document}